\documentclass[a4paper,12pt]{scrartcl}
\usepackage{amsmath,amssymb,amsthm}
\usepackage[all]{xy}
\usepackage{color}
\usepackage{hyperref}

\usepackage{tikz} 
\usetikzlibrary{matrix}

\usepackage{xspace, graphicx, epsfig, stmaryrd, dsfont}

\def\11{\mathds{1}}

\def\ZZ{\mathbb{Z}}

\def\FF{\mathbb{F}}

\def\bfA{\mathit{A}}

\newcommand{\h}{{\rm height}}
\newcommand{\F}{\mathcal{F}}
\newcommand{\Fr}{\mathrm{Fr}}
\newcommand{\bfa}{{\boldsymbol{a}}}

\def\bfA{\mathbf{A}}

\def\disc{\Delta} 

\def\iff{\Leftrightarrow}

\newtheorem{thm}{Theorem}[section]
\newtheorem{prop}[thm]{Proposition}
\newtheorem{lemma}[thm]{Lemma}

\theoremstyle{definition}

\theoremstyle{remark}

\def\Gal{{\rm Gal{}}}

\def\deg{{\rm deg{}}}
\def\disc{{\rm disc}}

\def\lll{$L(X)=f(t)+g(t)\cdot X\, $}

\begin{document}

\title{
Prime polynomial values of linear functions in short intervals
}

\author{Efrat Bank
\thanks{School of Mathematical Sciences, Tel Aviv University, Ramat Aviv, Tel Aviv 69978, Israel, \href{mailto:efratban@post.tau.ac.il}{efratban@post.tau.ac.il}}
\and  Lior Bary-Soroker \thanks{School of Mathematical Sciences, Tel Aviv University, Ramat Aviv, Tel Aviv 69978, Israel, \href{mailto:barylior@post.tau.ac.il}{barylior@post.tau.ac.il}}
}
\date{\today}

\maketitle
\begin{abstract}

In this paper we establish a function field analogue of a  conjecture in number theory which is 
a combination of several famous conjectures, including the Hardy-Littlewood prime tuple conjecture, conjectures on the number of primes in arithmetic progressions and in short intervals, and the Goldbach conjecture. We prove  an asymptotic formula for the number of simultaneous prime values of $n$ linear functions, in the limit of a large finite field. 
\end{abstract}

\section{Introduction}
Recently, several function field analogues of problems in analytic number theory were solved in the limit of a large finite field, e.g.\ the Bateman-Horn conjecture \cite{Alexei}; the Goldbach conjecture \cite{Goldbach_Bender2008}; the Chowla conjecture \cite{Mobius_CarmonRudnick2014}; problems on variance of the number of primes in short intervals and in arithmetic progressions \cite{KeatingRudnick2012} and on covariance of almost primes \cite{Rodgers}.

Let us describe in more detail two classical problems in number theory and their resolutions in the function field case. The two problems we describe relate to the work of this paper. 
Let $\11$ be the prime characteristic function, i.e.,
\begin{equation}
\11(h) = \begin{cases} 
1,& h \text{  is prime}\\
0, &\text{otherwise}.
\end{cases}\label{eq:charfn}
\end{equation}
The first problem is of counting primes in short intervals. 
By the Prime Number Theorem, it is conjectured that if $I$ is an interval of length $x^{\epsilon}$, $\epsilon>0$, around large number $x$, then 
\begin{equation}\label{eq:si}
\sum_{h\in I} \11(h) \sim \int_I \frac{dt}{\log t} \sim \frac{x^{\epsilon}}{\log x}.
\end{equation} 
Let $\FF_q[t]$ be the ring  of polynomials  over the finite field $\FF_q$ with $q$ elements.
By abuse of notation, denote by $\11$ the analogue of \eqref{eq:charfn}, i.e., the characteristic function of prime polynomials (which are by definition monic irreducibles), and let $\|f\|=q^{\deg f}$, for $f\in \FF_q[t]$ (where  $\|0\|=0$). 
Rosenzweig and the authors \cite{PrimePoly_Liors2014} prove the following analogue of \eqref{eq:si}: Let $f_0\in \FF_q[t]$ be monic of degree $k$, $\frac{3}{k}\leq \epsilon <1$, and $I=I(f_0,\epsilon)=\{ f\in \FF_q[t] : \|f-f_0\| \leq \|f_0\|^\epsilon\}$; then   
\begin{equation}\label{eq:Poly_SI}
\sum_{f\in I} \11(f) = \frac{\#I}{k} (1+O_k(q^{-1/2})),
\end{equation}
where the implied constant depends only on $k$ and not on $f_0$ or $q$. 
To compare between \eqref{eq:si} and \eqref{eq:Poly_SI}, we replace $x^{\epsilon}$ with $\#I$, and $\log x$ with $k$.

The second problem is the Hardy-Littlewood prime tuple conjecture, which asserts that 
\begin{equation}\label{eq:HL}
\sum_{0<h\leq x} \11(h+a_1) \cdots \11(h+a_n) \sim \frak{S}(a_1, \ldots, a_n) \frac{x}{(\log x)^n}, 
\qquad x\to \infty,
\end{equation}
where 
\[
\frak{S}(a_1, \ldots, a_n) = \prod_{p}\frac{1-\nu(p)p^{-1}}{(1-p^{-1})^{n}},
\]
and $\nu(p) = \#\{ h\mod p : (h+a_1)\cdots (h+a_n) \equiv 0 \pmod p\}$. Note that $\frak{S}=0$ if and only if $\nu(p)=p$ for some $p$, which implies that $p$ divides $(h+a_1)\cdots (h+a_n)$ for all $h$. Bender-Pollack \cite{QuantitativeGoldbach_BenderPollack2009} in the case $n=2$ and the second author \cite{HL_Lior2014} in general, prove that for any fixed $k>0$ 
\[
\sum_{\substack{f\in \FF_q[t] \text{ monic}\\ \deg f=k}} \11(f+a_1) \cdots \11(f+a_n) = \frac{q^{k}}{k^n} (1+O_{k,n}(q^{-1/2})),
\]
uniformly on all $a_1, \ldots, a_n\in \FF_q[t]$ of degrees $\deg (a_i)<k$ and for odd $q$. Recently, Carmon \cite{Carmon_Char2} proved this also in characteristic $2$.

Let us consider a more general setting that includes short intervals and arithmetic progressions. Let 
\[
L(X) = bX + a, \qquad a,b\in \ZZ
\]
be a linear function. Assume that $L$ is primitive in the sense that $\gcd(a,b)=1$. A folklore conjecture (see Section~\ref{s_heuristics}) states that, 
\begin{equation}\label{PNT_AP}
\sum_{h\in[x,x+x^{\epsilon}]} \11(L(h))\sim 
\frac{b}{\phi(b)}\cdot\frac{x^\epsilon}{\log(L(x))}, \qquad x\to \infty,
\end{equation}
where $\phi$ is the Euler totient function, and $0<a<b,\, b^\delta <x$ or $b<0,\, |b|^{1+\delta}<a$ and $|b|x^\alpha<a<|b|x^\beta$ for $1<\alpha<\beta$. 

Now, let $L_i=b_i X+a_i$, $i=1,\ldots, n$ be distinct primitive linear functions. 
As in the heuristic derivation of the Hardy-Littlewood conjecture from the Prime Number Theorem, one may expect (cf.\ \cite[Page 10]{icmGranville})  from \eqref{PNT_AP} that in the same range of parameters,
\begin{equation}\label{eq:conj}
\sum_{h\in [x,x+x^{\epsilon}]}\11(L_1(h))\cdots \11(L_n(h))\sim \mathfrak{S}(L_1, \ldots, L_n) \frac{x^{\epsilon}}{\prod_{i=1}^{n}\log(L_i(x))}, \qquad x\to \infty,
\end{equation}
where,  
\[
\frak{S}(L_1, \ldots, L_n) = \prod_{p}\frac{1-\nu_{L_1, \ldots, L_n}(p)p^{-1}}{(1-p^{-1})^n}
,
\]
and  $\nu_{L_1, \ldots, L_n}(p)$ is  the number of solutions $h\in \ZZ/p\ZZ$ of $\prod_i L_i(h)\equiv 0\pmod p$. Note that, $\frak{S}(L) = \frac{b}{\phi(b)}$, so \eqref{eq:conj} reduces to \eqref{PNT_AP} if $n=1$. 
%
%

It is interesting to notice that if we take $L_1(x)=x$ and $L_2(x) = a-x$, then \eqref{eq:conj} would imply a quantitative Goldbach conjecture (for  all sufficiently large even $a\in \ZZ$) and if we take $L_i(X) = x+a_i$, we retrieve \eqref{eq:HL}.

The aim of this study is to prove the function field analogue of \eqref{eq:conj} in the limit of a large finite field. For a primitive linear function $L(X) = f(t) + g(t)\cdot X$ with $f,g\in \FF_q[t]$ and $g\neq 0$, we define the height to be: $\h(L)=\max\{\deg (f), \deg (g)\}$.
Our main result is the following,

\begin{thm}\label{thm:main}
Let $B>0$ and $1> \epsilon > 0$ be fixed real numbers. Then the asymptotic formula
\[
\sum_{f\in I(f_0,\epsilon)} \11(L_1(f)) \cdots \11(L_n(f)) = \frac{\# I(f_0,\epsilon)}{\prod_{i=1}^n \deg(L_i(f_0))} (1+O_{B}(q^{-1/2}))
\]
holds uniformly for all odd prime powers $q$, $1\leq n\leq B$, distinct primitive linear functions $L_1(X), \ldots, L_n(X)$  defined over $\FF_q[t]$ each of height at most $B$, and monic $f_0\in \FF_q[t]$ of degree in the interval $B\geq \deg  f_0 \geq \frac{2}{\epsilon}$.
\end{thm}

Since $n$ is fixed and the $L_i$'s are all linear, $\frak{S}(L_1, \ldots, L_n) =1+O(q^{-1})$ (see \cite[1.3]{PrimeSpecialization_Pollack2008}). Hence, Theorem~\ref{thm:main} is indeed the analogue of \eqref{eq:conj} over $\FF_q[t]$ in the limit $q\to \infty$.
If $n=1$ and $b=1$, Theorem~\ref{thm:main} reduces to \cite[Corollary~2.4]{PrimePoly_Liors2014}; and if $n=1$ and $\epsilon >1-\frac{1}{B}$, it reduces to \cite[Corollary~2.6]{PrimePoly_Liors2014}. 

We not only count primes but also deal with other factorization types, see Theorem~\ref{thm:general}. The latter may be used to get independence of other arithmetic functions, e.g.\ the $k$-th divisor function $d_k(f) = \#\{(f_1,\ldots, f_k) : f=f_1\cdots f_k\}$, in parallel to \cite{ShiftedConvolution_LiorRudnick2014}.

The main innovation of the paper is the calculation of the Galois groups of certain polynomials, see \S\ref{s_discriminant}. The derivation of the main result from the Galois group calculation is then done in \S\ref{indthm}.

\section{A Galois group calculation}\label{s_discriminant}
The goal of this section is to calculate the Galois group of the product of $n$ linear functions evaluated at a generic polynomial.  \emph{For the rest of the section we fix an algebraically closed field $\FF$ of characteristic  not equal $2$}.

Recall that the discriminant of a monic separable polynomial $\mathcal{F}(t)$ is defined by the resultant of $\mathcal{F}$ and $\mathcal{F}'$:
\begin{equation}
\disc(\mathcal{F})=\pm  {\rm Res}(\mathcal{F},\mathcal{F}')=\pm\prod_{j=1}^\nu\mathcal{F}(\tau_j),
\end{equation}
where $\mathcal{F}' = c\prod_{j=1}^\nu(t-\tau_j)$. 

\begin{prop}\label{p_linearDis1}
Let $L_i(X)=f_i(t)+g_i(t)\cdot X$, where $i=1,2$ be distinct primitive linear functions over $\FF[t]$. Let $h=\sum_{j=0}^{m}A_jt^j$, where $A_j$ are variables. Assume that $\deg(f_i)>\deg(g_i)+m$, $i=1,2$ and that $m\geq 2$. Denote by $d_i=\disc_t(L_i(h))$ the discriminant of $L_i(h(t))$ regarded as a polynomial in $t$. Then $d_1$, $d_2$ are non-squares and are relatively prime in the ring $\FF(A_1, \ldots, A_m)[A_0]$. 
\end{prop}

Before proving the proposition we will prove three auxiliary results. 

\begin{lemma}\label{l_discPsi}
Let \lll be a primitive linear function and let $h(\mathbf{A},t)=\sum_{j=0}^{m}A_jt^j$ be a polynomial with variable coefficients.
Denote $\Psi(t)=\frac{f(t)}{g(t)}+\sum_{j=1}^{m}A_jt^j$. Assume that $(L(h))(\alpha)=0$ for some $\alpha$ in an algebraic closure $\Omega$ of $\FF(\mathbf{A})$. Then 
$$(L(h))'(\alpha)=0\iff \Psi'(\alpha)=0$$
\end{lemma} 

\begin{proof}
Write
$$(L(h))(t)=f(t)+g(t)\cdot \sum_{j=0}^m A_jt^j=(A_0+\Psi(t))\cdot g(t)$$
Since $f$ and $g$ are relatively prime, they do not have a common zero. Therefore, if $(L(h))(\alpha)=0$ then $g(\alpha)\neq 0$, which means that $A_0+\Psi(\alpha)=0$. Now,
$$(L(h))'(\alpha)=g'(\alpha)(A_0+\Psi(\alpha))+g(\alpha)\Psi'(\alpha)=g(\alpha)\Psi'(\alpha).$$
Hence, $(L(h))'(\alpha)=0\iff \Psi'(\alpha)=0$.
\end{proof}
\begin{lemma}\label{l_change sys of equasions}
Let $L_i(X)=f_i(t)+g_i(t)\cdot X$, $i=1,2$ be distinct primitive linear functions. Let $h(\mathbf{A},t)=\sum_{j=0}^{m}A_jt^j$, with $A_0, \ldots, A_m$ variables and let $\Psi_i(t)=\frac{f_i(t)}{g_i(t)}+\sum_{j=1}^{m}A_jt^j$. 
If $\rho_1,\rho_2$ in an algebraic closure of $\FF(\mathbf{A})$ solve the linear system 
\begin{equation}\label{eq:Lexcellentmorse}
\left\{
\begin{split}
&(L_1(h))'(\rho_1)=0\\
&(L_2(h))'(\rho_2)=0\\
&\Psi_1(\rho_1)=\Psi_2(\rho_2)=-A_0
\end{split}\right.
\end{equation}
then they solve the linear system
\begin{equation}\label{eq:excellentmorse}
\left\{
\begin{split}
\Psi_1'(\rho_1)&=0\\
\Psi_2'(\rho_2)&=0\\
\Psi_1(\rho_1)&=\Psi_2(\rho_2).
\end{split}\right.
\end{equation}
\end{lemma}
\begin{proof}
Note that since $A_0+\Psi_i(\rho_i)=0$, it follows that $(L_i(h))(\rho_i)=0$. Using Lemma~\ref{l_discPsi}, $\Psi_i'(\rho_i)=0$ for $i=1,2$.
\end{proof}

\begin{lemma}\label{l_linearComputation}
Let 
$2\leq m$ be an integer and let $L_1=f_1+g_1X,L_2=f_2+g_2X$ be distinct primitive linear functions over $\FF[t]$. Assume that $\deg(f_i)>\deg(g_i)+m,$ and let $h(\mathbf{A},t)=\sum_{j=0}^{m}A_jt^j$ be a polynomial with variable coefficients and $\Psi_i(t)=\frac{f_i(t)}{g_i(t)}+\sum_{j=1}^{m}A_jt^j$.
Then the linear system \eqref{eq:excellentmorse} has no solution for $\rho_1\neq\rho_2$ in an algebraic closure $\Omega$ of $\FF(\mathbf{A})$.
\end{lemma}

The proof is in the spirit of the proof of \cite[Lemma~3.5]{PrimePoly_Liors2014} and uses the tools of Carmon-Rudnick in \cite{Mobius_CarmonRudnick2014}. 
\begin{proof}
For short we write $\rho=(\rho_1,\rho_2)$.
Let
\[
-\varphi_i(t)=\bigg(\psi_i(t)+\sum_{j=3}^m A_jt^j\bigg)' =\psi_i(t)'+\sum_{j=3}^m jA_jt^{j-1},
\] 
where $i=1,2$ and $\psi_i=\frac{f_i}{g_i}$.
Let
\[
c(\rho) 
= \psi_1(\rho_1)-\psi_2(\rho_2) + \sum_{j=3}^m(\rho_1^{j}-\rho_2^{j})A_j.
\]
Then if $m\geq 3$,
\begin{align}
\Psi_i'(t)&=2A_2t+A_1-\varphi_i(t) \nonumber \\ 
c(\rho)&= \Psi_1(\rho_1)-\Psi_2(\rho_2) - ((\rho_1^2-\rho_2^2)A_2+(\rho_1-\rho_2)A_1).\label{eqmgeq3}
\end{align}

The system of equations \eqref{eq:excellentmorse} defines an algebraic set $T\subseteq \mathbb{A}^2\times \mathbb{A}^{m}$ in the variables $\rho_1, \rho_2, A_1, \ldots, A_m$.
It takes the matrix form 
\begin{equation}\label{eq:excellent}
M(\rho) \cdot \big(\begin{smallmatrix} A_2\\A_1\end{smallmatrix}\big) = B(\rho)=\Big(\begin{smallmatrix} \varphi_1(\rho_1)\\ \varphi_2(\rho_2)\\c(\rho)\end{smallmatrix}\Big), 
\end{equation}
where $M(\rho) = \left(\begin{smallmatrix}
2 \rho_1&1\\
2\rho_2&1\\
\rho_2^2-\rho_1^2&\rho_2-\rho_1
\end{smallmatrix}\right)$.

Let $\alpha\colon T\to \mathbb{A}^2$ and $\beta\colon T\to \mathbb{A}^m$ be the projection maps.
We note that since $\Psi_1\neq \Psi_2$, there are only finitely many solution of  \eqref{eq:excellentmorse} with $\rho_1= \rho_2$.

For every $\rho\in U=\{\rho\mid \rho_1\neq \rho_2,\ \varphi_i(\rho_i)\neq \infty, i=1,2\}$, the rank of $M(\rho)$ is $2$. Thus, the dimension of the fiber $\alpha^{-1}(\rho)$ is at most $m-2$ for any $\rho\in U$.
Moreover, for a given $\rho\in U$, \eqref{eq:excellent} is solvable if and only if $\mathop{\rm rank}(M|B)=2$ if and only if $d(\rho)=\det(M|B)=0$. This means that the solution space (restricting to $\rho\in U$) lies in $d(\rho)=0$. 

It suffices to prove that $d(\rho)$ is a nonzero rational function in the variables $\rho=(\rho_1,\rho_2)$. Indeed, this implies that  $\dim (\alpha(T))\leq \dim \{d(\rho)=0\} = 1$, so $\dim T \leq 1+m-2<m$. Thus, $\beta(T)$ does not contain the generic point of $\mathbb{A}^{m}$, which is $\mathbf{A}=(A_1,\ldots, A_m)$, and hence \eqref{eq:excellentmorse} has no solution with $\rho\in \Omega^2$. 

A straightforward calculation gives 
\[
d(\rho) = (\rho_1-\rho_2) (2c(\rho) +(\rho_1-\rho_2)(\varphi_1(\rho_1)+\varphi_2(\rho_2))).
\]
By \eqref{eqmgeq3}, if $m\geq 3$, then the coefficient of $A_3$ in $2c(\rho) +(\rho_1-\rho_2)(\varphi_1(\rho_1)+\varphi_2(\rho_2))$ is
\[
2(\rho_1^3-\rho_2^3)-3(\rho_1^2+\rho_2^2)(\rho_2-\rho_1),
\]
which is nonzero in any characteristic and we are done. 

Assume $m=2$. 
Then $c(\rho) = \psi_1(\rho_1)-\psi_2(\rho_2)$ and $\varphi_i(\rho_i)=-\psi_i'(\rho_i)$. So,
\[
d(\rho)=(\rho_1-\rho_2)(2(\psi_1(\rho_1)-\psi_2(\rho_2))-(\rho_1-\rho_2)(\psi_1'(\rho_1)+\psi_2'(\rho_2))).
\]

Assume that $d=0$ as a polynomial in $\rho$. Then also
\begin{equation}\label{eq:differential}
0=2(\psi_1(\rho_1)-\psi_2(\rho_2))-(\rho_1-\rho_2)(\psi_1'(\rho_1)+\psi_2'(\rho_2)).
\end{equation}
Let us solve \eqref{eq:differential} with $\psi_i$ rational function in $\rho_i$, $i=1,2$. 
Choose $\alpha$ such that $\psi_2(\alpha)=0$. By replacing $\rho_2$ by $\rho_2+\alpha$, we may assume that $\alpha=0$. Substituting $0$ for $\rho_2$ gives rise to the differential equation 
\begin{equation}\label{eq:psi1}
0=2\psi_1(\rho_1)-\rho_1\psi_1'(\rho_1)-\rho_1\psi_2'(0)
\end{equation}
As an element of the field of formal Laurent series, $\psi_1$ solving \eqref{eq:psi1} must have the form:
\begin{equation}\label{eq:psi1Sol}
\psi_1(\rho_1)=\rho_1\psi_2'(0)+\sum_{i=N}^{\infty}a_{ip+2}\rho_1^{ip+2}, \qquad N\in \mathbb{Z}.
\end{equation}
Plug \eqref{eq:psi1} and \eqref{eq:psi1Sol} in \eqref{eq:differential} to get 
\begin{align*}
0 &= 
\rho_1\psi_2'(0)  -2\psi_2(\rho_2) -\rho_1\psi_2'(\rho_2) +  \rho_2\left(\psi_2'(0)+\sum_{i=N}^{\infty}2a_{ip+2}\rho_1^{ip+1}\right)+\rho_2\psi_2'(\rho_2).
\end{align*}
Substituting $0$ for $\rho_1$ we get 
\begin{equation}\label{eq:psi2}
0=-2\psi_2(\rho_2)+\rho_2\psi_2'(0)+\rho_2\psi_2'(\rho_2),
\end{equation}
which is almost identical to  \eqref{eq:psi1Sol}; therefore, 
\begin{equation}\label{eq:psi2Sol}
\psi_2(\rho_2)=\rho_2\psi_2'(0)+\sum_{i=N}^{\infty}c_{ip+2}\rho_2^{ip+2}.
\end{equation}
Here, without loss of generality, we assume the series for $\psi_1$ and $\psi_2$ start at the same $N$, as we allow the coefficients to be zero.
Plug \eqref{eq:psi1Sol} and \eqref{eq:psi2Sol} in the original equation \eqref{eq:differential} to get,
\begin{equation}
0=2\left(\sum_{i=N}^{\infty}a_{ip+2}\rho_2\rho_1^{ip+1}-\sum_{i=N}^{\infty}c_{ip+2}\rho_1\rho_2^{ip+1}\right)
\end{equation}
By comparing the coefficients of $\rho_1\rho_2^{ip+1}$ and $\rho_2\rho_1^{ip+1}$, one gets that $a_{ip+2}=c_{ip+2}=0$  for all $i\neq 0$ and $a_2=c_2$. This means that 
\begin{align*}
\psi_1(\rho_1) &=\rho_1\psi_2'(0)+a_2\rho_1^2\\ 
\psi_2(\rho_2)&=\rho_2\psi_2'(0)+a_2\rho_2^2 
\end{align*}
in contradiction to the assumption that $\psi_i=\frac{f_i}{g_i}$ where $\deg (f_i) >2$ and $f_i,g_i$ are relatively prime. Therefore, $d(\rho)$ is not the  zero polynomial, as needed to conclude the proof.
\end{proof}

\begin{proof}[Proof of Proposition~\ref{p_linearDis1}]
By \cite[Proposition~3.6]{PrimePoly_Liors2014},  
\(\Gal(L_i(h),\FF(\mathbf{A}))\) is the full symmetric group. Hence, $d_i$ is not a square in $\FF(\mathbf{A})$ for each $i=1,2$ (otherwise, the group would be a subgroup of the alternating group) and in particular in $\FF(A_1,\ldots, A_m)[A_0]$. 
If $d_1$, $d_2$ are not relatively prime in $\FF(A_1,\ldots, A_m)[A_0]$, then they have a common root (as polynomials in $A_0$). 
Now,
\begin{align*}
d_1&= \disc_{t} L_1(h(t))=\pm \prod_{j=1}^{\nu}(L_1(h))(\tau_j)\\
&=\pm \prod_{j=1}^{\nu}g_1(\tau_j)(A_0+\Psi_1(\tau_j))
\end{align*}
where $(L_1(h))'(t) = c\cdot\prod_{j=1}^{\nu}(t-\tau_j)$. 
A root $\rho_1$ of $d_1$ must therefore satisfy:
\begin{equation}
\left\{
\begin{split}
&(L_1(h))'(\rho_1)=0\\
&\Psi_1(\rho_1)=-A_0
\end{split}\right.
\end{equation}
(note that if $g_1(\tau_j)=0$ then $(L_1(h))(\tau_j)\neq 0$). 
A root $\rho_2$ of $d_2$ satisfies the analogues equations. Thus, the condition that $d_1$ and $d_2$ have a common root translates into the linear system \eqref{eq:Lexcellentmorse}. By Lemma \ref{l_change sys of equasions},
the solutions for this system is a subset of the solutions of the linear system \eqref{eq:excellentmorse}, which is an empty set by Lemma \ref{l_linearComputation} whenever $\rho_1\neq \rho_2$. If $\rho_1=\rho_2$, then 
\[
\Psi_1(\rho_1)=\Psi_2(\rho_1)=-A_0
\] 
hence
\[
f_1(\rho_1)g_2(\rho_1)-f_2(\rho_1)g_1(\rho_1)=0.
\]
So $\rho_1$ is algebraic over $\FF$ in contradiction to $\Psi_1(\rho_1)=-A_0$.
Therefore, $d_1\cdot d_2$ is indeed not a square in  $\FF(\mathbf{A})$. Thus, $d_1$ and $d_2$ are relatively prime in $\FF(A_1,\ldots, A_m)[A_0]$.
\end{proof}

\begin{prop}\label{p_productGalois}
Let $L_1,\cdots,L_n$ be distinct primitive linear functions and $f_0\in \FF[t]$ a monic polynomial of degree $k$. Let $f=f_0+\sum_{j=0}^m A_jt^j$
where $2\leq m<k$. Then,
\[
\Gal\left(\prod_{i=1}^{n}L_i(f),\FF(\mathbf{A})\right)=\prod_{i=1}^{n}\Gal(L_i(f),\FF(\mathbf{A}))=S_{k_1}\times \cdots \times S_{k_n},
\]
where $k_i = \deg(L_i(f_0))$.
\end{prop}

\begin{proof}
Let $f=f_0+\sum_{j=0}^m A_jt^j$. Then, 
\[
L_i(f)= \tilde{L}_i(h)
\]
where $\tilde{L}_i=\tilde{f_i}+g_i\cdot X$, $h=\sum_{j=0}^m A_jt^j$, and $\tilde{f}_i = f_i + g_if_0$. 
Since $m<k$, it follows that $\deg (\tilde{f_i})>\deg (g_i) +m$ for each $i$. Since $m\geq 2$, \cite[Proposition~3.6]{PrimePoly_Liors2014} gives
\begin{equation}\label{eq:GG}
\Gal(L_i(f), \FF(\mathbf{A})) = \Gal(\tilde{L}_i(h),\FF(\mathbf{A}))\cong S_{k_i}
\end{equation}
By Proposition~\ref{p_linearDis1}, the discriminants $d_i=\disc_t(\tilde{L}_i(h))=\disc_t(L_i(f))$ are non-squares and pairwise relatively prime in $\FF(A_1, \ldots,A_m)[A_0]$. 
So, $d_1, \ldots, d_n$ are square independent (in the sense that any product is non-square). Together with \eqref{eq:GG}, the discussion before \cite[Lemma~3.4]{IrrVal_Lior2012}  gives that 
\[
\Gal\left(\prod_{i=1}^{n}L_i(f),\FF(\mathbf{A})\right)=\prod_{i=1}^{n}\Gal(L_i(f),\FF(\mathbf{A}))=S_{k_1}\times \cdots \times S_{k_n},
\]
as needed. 
\end{proof}

\section{Independence theorem}\label{indthm}
In this section we shall prove a generalization of Theorem~\ref{thm:main}. 

We follow the notation of \cite{ShiftedConvolution_LiorRudnick2014}. 
The cycle structure of a permutation $\sigma$ of $k$ letters is the partition 
$\lambda(\sigma) = (\lambda_1,\dots, \lambda_k)$ of $k$ if in the
decomposition of $\sigma$ as a product of disjoint cycles, there are
 $\lambda_j$ cycles of length $j$.

For each partition $\lambda \vdash k$, the probability that a random permutation on $k$ letters has cycle structure $\lambda$ is given by Cauchy's formula:  
\begin{equation}\label{def p}
p(\lambda) = \frac{\#\{\sigma\in S_k: \lambda(\sigma)=
\lambda\}}{\#S_k } = \prod_{j=1}^k \frac 1{j^{\lambda_j} \cdot \lambda_j!}.
\end{equation}

For $f\in \FF_q[t]$ of positive degree $k$, we say its cycle structure
is $\lambda(f) = (\lambda_1,\dots, \lambda_k)$ 
if in the prime decomposition $f=\prod_j P_j$ (we allow repetition), 
we have $\#\{i: \deg(P_i)=j\} = \lambda_j$. 

For a partition $\lambda\vdash k$, we let $\11_{\lambda}$ be the characteristic function of $f\in \mathcal{M}_n$ of cycle structure $\lambda$:
\begin{equation}
\11_\lambda(f)=\begin{cases} 1, & \lambda(f) = \lambda\\ 0, &\mbox{otherwise}.\end{cases}
\end{equation}
%

\begin{thm}\label{thm:general}
Let $B>0$ and $1>\epsilon> 0$ be fixed real numbers. Then the asymptotic formula
\[
\sum_{f\in I(f_0,\epsilon)} \11_{\lambda_1}(L_1(f)) \cdots \11_{\lambda_n}(L_n(f))
= p(\lambda_1) \cdots p(\lambda_n) \#I(f_0,\epsilon)\left(1+O_B\left(q^{-\frac{1}{2}}\right)\right)
\]
holds uniformly for all odd prime powers $q$, $1\leq n \leq B$, distinct primitive linear functions $L_1(X), \ldots, L_n(X)$ defined over $\FF_q[t]$ each of height at most $B$, monic $f_0\in \FF_q[t]$ of degree in the interval $B\geq \deg f_0 \geq \frac{2}{\epsilon}$, and partitions $\lambda_1, \cdots ,\lambda_n$ of $\deg(L_1(f_0)), \ldots, \deg (L_n(f_0))$, respectively.
\end{thm}

Theorem~\ref{thm:main} is the special case of Theorem~\ref{thm:general} when taking all the $\lambda_i$'s to be the partition into one part. (Since then $\11 = \11_{\lambda_i}$ and $p(\lambda_i) = \frac{1}{\deg(L_i(f_0))}$.)

The proof of Theorem~\ref{thm:general} is in the same spirit as proofs of other results in the literature once one has the Galois group calculation (Proposition~\ref{p_productGalois}). In fact, it is nearly identical to the proof of \cite[Theorem~1.4]{ShiftedConvolution_LiorRudnick2014}. For the reader's convenience we bring here the full proof. 
\begin{thm}[{\cite[Theorem~3.1]{ShiftedConvolution_LiorRudnick2014}}] \label{thm:cheb} 
Let $\bfA=(A_0, \ldots, A_m)$ be an $(m+1)$-tuple of variables over $\FF_q$, let $\F(t) \in \FF_q[\bfA][t]$ be monic and separable in $t$, let $L$ be a splitting field of $\F$ over $K=\FF_q(\bfA)$, and let $G=\Gal(\F,K)=\Gal(L/K)$. Assume that $\FF_q$ is algebraically closed in $L$. 
Then there exists a constant $c=c(m,{\rm tot.deg} (\F))$ such that for every conjugacy class $C\subseteq G$ we have 
\[
\left|\# \{ \bfa\in \FF_q^{m+1} : \Fr_\bfa = C\} -\frac{|C|}{|G|} q^{m+1}\right|\leq c q^{m+1/2}.
\]
\end{thm}

Here $\Fr_{\bfa}$ denotes the Frobenius conjugacy class $\left(\frac{S/R}{\phi}\right)$ in $G$ associated to the homomorphism  $\phi \colon R\to \FF_q$ given by $\bfA\mapsto \bfa\in \FF_q^{m+1}$, where $R=\FF_q[\bfA,\disc\F^{-1}]$ and $S$ is the integral closure of $R$ in the splitting field of $\F$. See \cite[Appendix A]{ShiftedConvolution_LiorRudnick2014} for more details and for a proof.

\begin{proof}[Proof of Theorem \ref{thm:general}.]
Let $\epsilon>0$ and $f_0$ be a monic polynomial of degree $k$ where $B\geq k \geq \frac{2}{\epsilon}$. Set $m=\lfloor\epsilon k \rfloor$ and let $f=f_0+\sum_{j=0}^{m}A_jt^j$. Note that $2\leq m<k$. Define $\F_i(\mathbf{A},t)=L_i(f)$ and $\F=\F_1\cdots\F_n$. Let $L$ be the splitting field of $\F$ over $K=\FF_q(\mathbf{A})$ and let $\FF$ be an algebraic closure of $\FF_q$. 
Let $G=\Gal(\F,K) = \Gal(L/K)$. 

By Proposition~\ref{p_productGalois}, 
\[
S_{k_1}\times \cdots \times S_{k_n}\cong \Gal(\FF L/\FF K)\cong\Gal(L/L\cap (\FF K)) \leq G,
\]
where $k_i = \deg(L_i(f_0))$.
On the other hand, the factorization $\F=\F_1\cdots\F_n$ implies that $G\leq S_{k_1}\times \cdots \times S_{k_n}$. So 
\begin{equation}\label{Eq:PP}
G=S_{k_1}\times \cdots \times S_{k_n},
\end{equation}
and $L\cap (\FF K)=K$.  It follows in particular that $L\cap \FF=K\cap \FF=\FF_q$. Hence, we may apply Theorem~\ref{thm:cheb} with the conjugacy class 
\[
C=\{ (\sigma_1, \ldots, \sigma_n)\in G : \lambda_{\sigma_i} = \lambda_i\}
\]
to get that 
\[
\left|\#\{\bfa\in \FF_{q}^{m+1} : \Fr_{\bfa} = C \} - |C|/|G| \cdot q^{m+1}\right|\leq c(B) q^{m+1/2}.
\]
We note that $|C|/|G| = p(\lambda_1) \cdots p(\lambda_n)$ and $\#\{\bfa\in \FF_{q}^{m+1} : \disc_t(\F)(\bfa)=0\} = O_{B}(q^{m})$. Also for $\bfa\in \FF_q^{m+1}$ with $\disc_t(\F(\bfa))\neq 0$ we have $\Fr_{\bfa} =C$ if and only if $\lambda_{\F_i(\bfa,t)}=\lambda_i$ for all $i=1, \ldots, n$ (see the proof of \cite[Theorem~3.1]{ShiftedConvolution_LiorRudnick2014} where this is shown explicitly).
Now,
\begin{align*}
&\sum_{f\in I(f_0,\epsilon)}\11_{\lambda_1}(L_1(f)) \cdots \11_{\lambda_n}(L_n(f)) \\
&\qquad = \#\{\bfa\in \FF_{q}^{m+1} : \lambda_{\F_i(\bfa,t)}=\lambda_i \text{ for all } i\}\\
	&\qquad = \#\{\bfa\in \FF_{q}^{m+1} : \disc_t(\F)(\bfa)\neq 0, \lambda_{\F_i(\bfa,t)}=\lambda_i \text{ for all } i\} + O_{B}(q^m)\\
	& \qquad = \#\{\bfa\in \FF_{q}^{m+1} : \Fr_{\bfa} = C \} + O_{B}(q^m)\\
	& \qquad = |C|/|G | q^{m+1} + O_{B}(q^{m+1/2}) \\
	& \qquad = p(\lambda_1) \cdots p(\lambda_n) q^{m+1} (1+O_B(q^{-1/2})).
\end{align*}
This finishes the proof, since $\#I(f_0,\epsilon) = q^{m+1}$.
\end{proof}

\section{A discussion on primes in short intervals and in arithmetic progressions}\label{s_heuristics}
In this section we provide heuristic for \eqref{PNT_AP}. Fix $1>\epsilon>0$.
A classical conjecture about primes in short intervals of the form $[x,x+x^{\epsilon}]$ asserts that
\begin{equation*}
\sum_{x\leq h\leq x+x^{\epsilon} }\11(h)\sim \int_{x}^{x+x^\epsilon} \frac{dt}{\log t}\sim \frac{x^{\epsilon}}{\log (x)}, \qquad x\to \infty.
\end{equation*}
Another classical conjecture, on the number of primes in arithmetic progressions, says that if  $0<a<b<x^{1-\delta}$, then
\begin{equation*}
\sum_{
\substack{0<h<x\\ h\equiv a\pmod b}
}\11(
h)\sim \frac{1}{\phi(b)}\cdot\frac{x}{\log(x)}, \qquad x\to \infty.
\end{equation*}
If $\epsilon+\delta \leq 1$, then the number of $h\equiv a\pmod b$ in $[x,x+x^{\epsilon}]$ is at most $1$. Therefore, 
to combine these two conjectures together, we must at least demand that $1<\epsilon+\delta$. Then it is natural to expect that: 

\begin{equation}\label{conj-comb}
\sum_{
\substack{x\leq h\leq x+u\\ h\equiv a\pmod b}
}\11(
h)\sim \frac{1}{\phi(b)}\cdot\frac{u}{\log(x)}, \qquad x\to \infty.
\end{equation}
for $0<a<b<x^{1-\delta}$ and $x^{\epsilon}\leq u\leq x$ where $\epsilon>0$.
%

In our setting it is more convenient to reformulate \eqref{conj-comb} as a statement on the mean value of $L(X)=a+bX$:
\begin{equation}\label{eq:conjbgeq0}
\begin{split}
\sum_{x\leq h\leq x+x^{\epsilon}}\11(L(h))
	&=\sum_{x\leq h\leq x+x^{\epsilon}}\11(a+bh)\\
	&=\sum_{\substack{a+y\leq \tilde{h}\leq a+y+b^{1-\epsilon}y^{\epsilon}\\\tilde{h}\equiv a\pmod b}}\11(\tilde{h})\\
	&\sim \frac{1}{\phi(b)}\cdot\frac{b^{1-\epsilon}y^\epsilon}{\log(y)} \sim
	\frac{1}{\phi(b)}\cdot\frac{bx^{\epsilon}}{\log(L(x))},
\end{split}
\end{equation}
where $y=bx$ and $b^{\delta}<x$ and $0<a<b$. 
(Note that when $b^{\delta}<x$, then  $b<y^{1-\delta'}$ for some $\delta'>0$, which implies that $y^{\epsilon} \leq b^{1-\epsilon}y^{\epsilon} \leq y$; therefore the prior to the last step is justified.)

Next we deal with the case where $b<0$ and $a>|b|$. By division with remainder, there are unique $0< a'<|b|$ and $0<r$ such that $a=a'+|b|r$. Thus $a+bh = a' +|b|(r-h)$. 
So, putting $h'=r-h$ and $L'(X)=a'+|b|X$, we get 
\begin{equation*}
\begin{split}
\sum_{x\leq h\leq x+x^{\epsilon}}\11(L(h))
			&=\sum_{r-x-x^{\epsilon}\leq h'\leq r-x}\11(L'(h'))\\
			&=\sum_{y\leq h' \leq y+x^\epsilon} \11(L'(h'))	\\		
			&\sim \frac{1}{\phi(|b|)}\cdot\frac{|b|x^{\epsilon}}{\log(L'(y))}\\
			&\sim \frac{1}{\phi(|b|)}\cdot\frac{|b|x^{\epsilon}}{\log(L(x))},
\end{split}
\end{equation*}
where $y=r-x-x^\epsilon$ and $|b|^{1+\delta}<a$ and $|b|x^{\alpha}<a<|b|x^{\beta}$ for $1<\alpha<\beta$.
(Note that when these conditions hold, $y^{\epsilon'}<x^{\epsilon}<y$ for $\epsilon'<\frac{\epsilon}{\beta}$ and $|b|^\delta <y$ since $y\sim \frac{a}{|b|}$; therefore the prior to the step is justified.) 

\bibliographystyle{plain}

\end{document}